\documentclass[11pt,reqno]{article} 

\usepackage[utf8]{inputenc}
\usepackage{graphicx}
\usepackage[all]{xy}
\usepackage{amssymb}
\usepackage{amsmath} 
\usepackage{amsthm}
\usepackage{verbatim}
\usepackage{latexsym}
\usepackage{mathrsfs}

\usepackage{mathrsfs}

\makeindex

\renewcommand{\Cup}{\bigcup}
\renewcommand{\Cap}{\bigcap}

\renewcommand{\a}{\alpha}
\renewcommand{\b}{\beta}
\newcommand{\g}{\gamma}
\newcommand{\e}{\varepsilon}
\renewcommand{\d}{\delta}
\renewcommand{\k}{\kappa}

\newcommand{\f}{\varphi}
\renewcommand{\o}{\omega}

\newcommand{\es}{\varnothing}

\newcommand{\A}{\mathcal{A}}

\newcommand{\E}{\mathcal{E}}

\newcommand{\fatdiamond}{\text{\rotatebox[origin=c]{90}{$\diamondsuit$}}}

\newcommand{\la}{\langle} 
\newcommand{\ra}{\rangle}
\newcommand{\Po}{\mathcal{P}}             
 
\newcommand{\sing}{{\operatorname{sing}_\omega}}
\newcommand{\ZFC}{{\operatorname{ZFC}}}

\newcommand{\SH}{{\operatorname{SH}}}

\newcommand{\cf}{\operatorname{cf}}

\newcommand{\id}{\operatorname{id}}

\newcommand{\height}{{\operatorname{ht}}}
\newcommand{\pr}{\operatorname{pr}}
\newcommand{\rest}{\!\restriction\!}
\newcommand{\restl}{\restriction}  
\newcommand{\dom}{\operatorname{dom}}

\renewcommand{\le}{\leqslant}  
\renewcommand{\ge}{\geqslant}  
  
\newcommand{\sd}{\,\triangle\,}             

\newcommand{\PlOne}{\,{\textrm{\bf I}}}
\newcommand{\PlTwo}{\textrm{\bf I\hspace{-1pt}I}}

\swapnumbers
\newtheorem*{Thm*}{Theorem}
\newtheorem{Thm}{Theorem}
\newtheorem{Lemma}[Thm]{Lemma}

\newtheorem{claim}{Claim}[Thm]

\theoremstyle{definition}
\newtheorem{Def}[Thm]{Definition}

\theoremstyle{remark}
\newtheorem*{Remark}{Remark}

\newcommand{\proofvpara}{\text{}}

\author{Sy-David Friedman, Vadim Kulikov\\ Kurt Gödel Research Center\\University of Vienna}
\title{Failures of the Silver Dichotomy in the Generalised Baire Space}
\date{February 2014}

\begin{document}

\maketitle

\begin{abstract}
  We prove results that falsify Silver's dichotomy for Borel equivalence relations
  on the generalised Baire space under the assumption~$V=L$.
\end{abstract}

\vspace{10pt}
\vspace{10pt}

\section{Introduction}
\label{sec:intro}

The study of Borel equivalence relations and their reducibility springs from
the interest in classification problems in mathematics. The classical theory
studies Borel and analytic equivalence relations on Polish spaces and the partial
order formed by these equivalence relation with respect to Borel reducibility $\la\E,\le_B\ra$.
The generalised descriptive theory, initiated in the 1990's by the work of Halko, Mekler, Väänänen
and Shelah \cite{Hal,HalShe,MekVaa}, and recently developed further \cite{FHK,Kul,Lucke} studies
the classification problems on generalised Baire and Cantor spaces, $\k^{\k}$ and $2^\k$ for
uncountable regular $\k$. As is already a custom we concentrate on cardinals with $\k^{<\k}=\k$. 

We show that the classical result, known as the Silver dichotomy, fails in the generalised
setting in the following two ways.
It was shown in \cite{Kul}, in particular, that the power set of $\k$ ordered by inclusion,
$\la\Po(\k),\subset\ra$ can be embedded into $\la\E,\le_B\ra$.
In this paper we show that if $\k$ is inaccessible and $V=L$, then 
$\la\la\Po(\k),\subset\ra$ can be embedded into $\la\E,\le_B\ra$ \emph{below}
the identity relation (Theorem~\ref{thm:Main1}). Then we show that if $V=L$ and $\k$ is uncountable
and regular, then there is an antichain with respect to $\le_B$ of length $2^{\k}$ of Borel equivalence relations
and each of these relations is also incomparable with the identity relation. 

In this paper we always work in $\ZFC+V=L$ unless stated otherwise.

\paragraph{Acknowledgement}

The authors wish to thank the FWF for its support through Project P 24654 N25.

\section{Preliminaries}

\subsection{Fat Diamond}
\label{ssec:FatDiamond}

\begin{Def}[Fat diamond] 
  A \emph{fat diamond} on $\k$, denoted $\fatdiamond_\k$, is a sequence $\la S_\a\mid \a<\k\ra$ such that
  for all $\a<\k$, $S_\a\subset \a$ and for every set $S\subset\k$, cub set $C\subset \k$ and
  $\g<\k$ there is a continuous increasing sequence of order type $\g$ inside the set 
  $$C\cap \{\a<\k\mid S\cap \a=S_\a\}.$$
\end{Def}

\begin{Thm}[$V=L$] 
  A $\fatdiamond_\k$-sequence exists for uncountable regular $\k$.  
\end{Thm}
\begin{proof}
  Suppose $\b\le\k$, $(S_\a)_{\a<\b}$ is defined and
  $(C^*,S^*,\g^*)$ is a triple such that $C^*,S^*\subset \b$, $\g^*<\b$,
  $C^*$ is cub in $\b$ and there
  exists no continuous increasing sequence of order type $\g^*$ in $C^*\cap \{\a\mid S^*\cap\a=S_\a\}$.
  Then we abbreviate this by $R(C^*,S^*,\g^*)$.

  Let $(C_0,S_0)=(\es,\es)$. By induction, suppose that $(C_\a,S_\a)$ is defined for all $\a<\b$.
  Let $(C_\b,S_\b)$ be the $L$-least pair such that for some $\g<\b$ we have $R(C_\b,S_\b,\g)$,
  if such exists and set $(C_\b,S_\b)=(\es,\es)$ otherwise. 

  Let us show that the sequence $\la S_\a\ra_{\a<\k}$ obtained in this way is a $\fatdiamond_\k$-sequence. 
  Note that this sequence is definable in~$L$.
  Suppose on the contrary that it is not a $\fatdiamond_\k$-sequence. Then
  there exists a cub $C$ and $S\subset \k$ such that there exists an ordinal $\g<\k$
  with~$R(C,S,\g)$.
  Suppose that $(C,S)$ is the $L$-least such pair and $\g$ the least ordinal 
  witnessing this. Note that then $(C,S)$ and $\g$ are definable in~$L$.
  Then build a continuous increasing sequence $(M_\b)_{\b<\g}$ of elementary
  submodels of $L_{\k^+}$
  such that $(M_\b\cap\k)_{\b<\g}$ is a continuous increasing
  sequence of ordinals in~$C$.
  Let us show that each $M_\b\cap\k$ is also in $\{\a\mid S\cap\a=S_\a\}$ which is a contradiction.

  So let $\b<\g$, denote $\b'=M_\b\cap \k$ and let $\pi$ be the transitive
  collapse of $M_\b$ onto some $L_{\b''}$. 
  Then $\pi(C)=C\cap\b'$ and $\pi(S)=S\cap\b'$. Moreover by the elementarity of $M_\b$ in $L$
  $$M\models (C,S)\text{ is the }L\text{-least pair s.t. }\exists\d<\k R(C,S,\d).$$
  Applying $\pi$ we get
  $$L_{\b''}\models (C\cap\b',S\cap \b')\text{ is the }L
  \text{-least pair s.t. }\exists\d<\b'R(C\cap\b',S\cap \b',\d)$$
  But then by absoluteness of the $L$-ordering, this holds also in $L$, so
  in fact $S\cap\b'=S_{\b'}$, so $\b'$ is in $\{\a\mid S\cap\a=S_\a\}$ as intended.
\end{proof}

\begin{Def}
  A stationary set $S\subset \k$ is \emph{fat}, if for all cub sets $C\subset\k$
  and all $\g<\k$ there is a continuous increasing sequence of length $\g$ in $S\cap C$.
\end{Def}

\begin{Thm}[$V=L$]\label{thm:fatsetInL}
  If $\k>\o$ is regular, then there exists a fat stationary set $S\subset\k$ such that $\k\setminus S$
  is also fat stationary.
\end{Thm}
\begin{proof}
  Let $S=\{\a\mid S_\a=\a\}$ where $(S_\a)_{\a<\k}$ is the $\fatdiamond_\k$-sequence
  defined above. Then $S$ is fat by definition.
  But also $S'=\{\a\mid S_\a=\es\}$ is fat stationary and is disjoint from~$S$.
\end{proof}

\subsection{Trees}
\label{sec:Trees}

\begin{Def}
  $\k^{<\k}$ is the tree consisting of all functions $p\colon\a\to\k$ for $\a<\k$ ordered by end-extension:
  $p<q\iff p\subset q$.
  By a \emph{tree} we mean a downward closed suborder of $\k^{<\k}$. 
  A subtree is a subset of a tree which is itself a tree.
  Let $T\subset \k^{<\k}$ be a tree. A \emph{branch} through
  $T$ is a set $b$ which is a maximal linear suborder of $T$. The set of all branches of $T$ is denoted by
  $[T]$. The set of all branches of length $\k$ is denoted by $[T]_{\k}$. 
  The \emph{height} of an element $p\in T$, denoted $\height(p)$, is the order type of $\{q\in t\mid q<p\}$.
  Let $\a<\k$ be an ordinal. Denote by $T^{\a}$ the subtree of $T$ formed
  by all the elements with $\height(p)<\a$.  
\end{Def}

\section{For an Inaccessible $\k$.}
\label{sec:EmbeddingOfPk}

In this section we show that, if $V=L$ and $\k$ is a strongly inaccessible cardinal,
then there exists an embedding $F$ of $\la \Po(\k),\subset\ra$ into $\la\E,\le_B\ra$ where
$\E$ is the set of Borel equivalence relations on $2^\k$ such that
for all $A\in \Po(\k)$, $F(A)\lneq_B \id_{2^\k}$.

\begin{Def}\label{def:Trees}
  Let $\sing(\k)$ be the set of singular $\omega$-cofinal cardinals below $\k$.
  We will construct for each set $S\subset \sing(\k)$ a \emph{weak $S$-Kurepa tree} $T_S$ as follows.
  For each $\a\in\sing(\k)$ let $f(\a)$ be the least limit ordinal $\b$ such that 
  $L_{\b}\models(\a\text{ is singular})$.
  Then let 
  $$T_S=\{s\in 2^{<\k}\mid \forall \a\le\dom(s)(\a\in S\rightarrow s\rest \a\in L_{f(\a)})\}.$$
\end{Def}

\begin{Lemma}\label{lem:SmallLevels}
  Let $S\subset\sing(\k)$. Then for every $\g\in S$ we have $|T_S^{\g+1}|=|\g|$.
\end{Lemma}
\begin{proof}
  When $\g\in S$, then $T_S^{\g+1}$ is a subset of $L_{f(\g)}$ whose cardinality is
  $|f(\g)|$. But $|f(\g)|<|\g|^+$, so we are done.
\end{proof}

\begin{Remark}
  For the converse, 
  if $S$ is as in Lemma \ref{lem:SmallLevels} and $\g\in S^\k_\o\setminus S$, then
  we have $|T^{\g+1}_S|=2^{\g}$: by GCH and $\cf\g=\o$, we have $|2^{\g}|=|\g^\o|$ and
  it is enough to consider increasing cofinal sequences in $\g$. But if $s$ is an increasing
  cofinal sequence in $\g$ of length $\o$, then it is an element of $T_S$.
\end{Remark}

\begin{Lemma}\label{lem:KurepaCo-meager}
  Let $T=T_S$ be a tree as above for some $S\subset\sing(\k)$.
  Let $(D_i)_{i<\k}$ be a sequence of dense open subsets of $T$, i.e.
  such that for all $i<\k$ we have $\forall p\in T\exists q\in D_i(q>p)$ (density)
  and $\forall p\in D_i(N_p\cap T\subset D_i)$ (openness).
  Then there is a branch of length $\k$ in $\Cap_{i<\k}D_i$ through $T$.
\end{Lemma}
\begin{proof}
  Suppose there is a sequence $(D_i)_{i<\k}$ of dense open subsets of $[T_S]_\k$ such that
  $\Cap_{i<\k}D_i$ is empty. Suppose $(D_i)_{i<\k}$ is the $L$-least such sequence.
  For a contradiction it is enough to find a branch through $T_{\sing(\k)}$ in $\Cap_{i<\k} D_i$.
  Let $(M_{\g})_{\g<\k}$ be
  a definable 
  continuous increasing sequence of sufficiently 
  elementary submodels of $(L_{\k^+},\in)$ of size $<\k$
  such that $M_\g\cap\k=\g'$
  for some $\g'<\k$ and $M_\g$ contains a Borel code for $D_\g$.
  Let $L_{\g''}$ be the result of the transitive collapse of $M_\g$. Now pick the $L$-least
  $p_0\in L_{0''}$ such that
  $L_{0''}\models (N_{p_0}\subset D_0\land p_0\in T_{\sing(\k)})$. 
  Note that this implies that $p_0\in T_{\sing(\k)}$.
  If $p_\g$ is defined to be an element of $L_{\g''}$, let $p_{\g+1}$
  be the $L$-least element of $L_{(\g+1)''}\cap T_S$ extending $p_\g$ such that
  $L_{(\g+1)''}\models (N_{p_{\g+1}}\subset D_{\g+1}\land p_{\g+1}\in T_{\sing(\k)})$     
  and $\dom p_{\g+1}>\g'$. If $\g$ is a limit and $p_\b$
  are defined for all $\b<\g$, then let $p_{\g}=\Cup_{\b<\g}p_\b$. The sequence $(M_{\b})_{\b<\g}$ is definable
  in $L_{\g''}$, and so is $p_\g$. On the other hand $\dom p=\g'$ is regular from the viewpoint of $L_{\g''}$,
  so $p_\g\in T_{\sing(\k)}$. In this way we obtain a branch through $T_{\sing\k}\subset T_S$
  in $\Cap_{i<\k}D_i$. 
\end{proof}

\begin{Lemma}
  For every $S\subset \sing(\k)$, $T_S$ has $\k^+$ branches of length $\k$, i.e. \mbox{$|[T_S]_{\k}|=\k^+$}.
\end{Lemma}
\begin{proof}
  As remarked above, $T_{\sing(\k)}\subset T_S$, so
  it is sufficient to show that $T_{\sing(\k)}$ has $\k^+$ branches.
  For each $\b<\k^+$ let $C(\b)=\{\g<\k\mid \SH^{L_{\b}}(\g \cup\{\k\})\cap\k=\g\}$.

  We want to show that there is an unbounded set $G\subset\k^+$ such that 
  for all $\b,\b'\in G$ the sets $C(\b)$ and $C(\b')$ are all different if $\b\ne\b'$
  and that the characteristic function of each $C(\b)$ is a branch through $T_{\sing(\k)}$.
  We claim that $G=\{\b<\k^+\mid \SH^{L_\b}(\k\cup\{\k\})=L_\b\}$ is such a set. 
  To show that $G$ is unbounded, let $\b<\k^+$ and let $X\subset \k$ be a set
  such that $X\notin L_\b$. Let $\b'<\k^+$ be the least ordinal
  such that $X$ is definable in $L_{\b'}$ with parameters, so $\b'\ge \b$.
  Let $\f(p)$ be a formula with parameters $p$,
  which defines $X$ and let $p_0$ be the $L$-least sequence of parameters such that
  $\f(p_0)$ defines a subset of $L_{\b'}$ which is not an element of $L_{\b'}$. Now $p_0$ is in
  $\SH^{L_{\b'}}(\k\cup\{\k\})$. Let $\bar\b$ be such that
  $\SH^{L_{\b'}}(\k\cup\{\k\})\cong L_{\bar\b}$. We want to show that $\bar\b=\b'$.
  But since $p_0\in\SH^{L_{\b'}}(\k\cup\{\k\})$, the set defined by $\f(p_0)$ in $L_{\bar\b}$
  is in $L_{\bar\b+1}$. But by the definition of $p_0$, this set cannot be in $L_{\b'}$, so
  $\bar\b=\b'$.
  
  Suppose $\b,\b'\in G$ and $\b<\b'$. We claim 
  that $C(\b')\subset^* \lim C(\b)$, where $\subset^*$ means inclusion modulo
  a bounded set and $\lim$ denotes the limit points of a set. This clearly implies that $C(\b)\ne C(\b')$.
  Suppose $\g\in C(\b')$. Then since $\b'\in G$, we have $\b\in \SH^{L_{\b'}}(\g\cup\{\k\})$ for
  any $\g$ greater than some~$\g^*<\k$. Now every Skolem function of $L_{\b}$ is definable in $L_{\b'}$
  with parameters from $\g\cup \{\k\}$, so $\b'\in C(\b)$. But in fact, also $C(\b)$ is definable
  in $L_{\b'}$ with these parameters, so in fact $\b'\in \lim C(\b)$. 
  Thus $C(\b')\setminus\g^*\subset \lim C(\b)$.

  Let $f_\b$ be the characteristic function of $C(\b)$ and let us show that $(f_\b\rest\a)_{\a<\k}$
  is a branch of $T=T_{\sing(\k)}$. By the definition of $T$ it is sufficient to show that
  $f_\b\rest\a$ is in $L_{f(\a)}$ for all singular $\a\in \k$; this is of course equivalent
  to $C(\b)\cap\a$ being in $L_{f(\a)}$. There are two cases: either $\a\in C(\b)$ or $\a\notin C(\b)$.
  If $\a$ is in $C(\b)$, then by the definition
  of $C(\b)$, $\SH^{L_\b}(\a\cup\{\k\})\cap\k=\a$. Let $\bar\b$ be such that
  $L_{\bar\b}$ is the transitive collapse of $\SH^{L_\b}(\a\cup\{\k\})$. Then
  $C(\b)\cap\a\in L_{\bar\b+2}$ and since $\k$ becomes $\a$ in the collapse, $\a$ is
  regular in $L_{\bar\b}$ and so $\bar\b<f(\a)$. But $f(\a)$ was chosen to be a limit ordinal,
  $\bar\b+2<f(\a)$ as well. Thus $C(\b)\cap\a\in L_{f(\a)}$.
  Suppose that $\a\notin C(\b)$. But then $C(\b)\cap \a$ is bounded in $\a$ and since $\a$
  is a cardinal, $C(\b)\cap\a\in L_{\a}\subset L_{f(\a)}$.
\end{proof}

\begin{Thm}\label{thm:Main1}
  Suppose $V=L$ and $\k$ is inaccessible. Then the order $\la\Po(\k),\subset\ra$ can be embedded
  into $\la \E,\le_B\ra$ (Borel equivalence relations) strictly below the identity on~$2^\k$. 
  More precisely, there exists $F\colon \Po(\k)\to \E$ such that
  for all $A_0,A_1\subset \Po(\k)$ we have 
  $A_0\subset A_1\iff F(A_0)\le_B F(A_1)$ and $F(A_0)\lneq_B\id_{2^\k}$.
\end{Thm}
\begin{proof}


  For a tree $T\subset 2^{<\k}$ let $E(T)$ be the equivalence relation on $2^{\k}$
  such that two elements are equivalent if and only if both of them are not branches of $T$
  \emph{or} they are identical. 

  \begin{claim}\label{claim:Something}
    Suppose $S_0\subset \k$ is a fat stationary set such that $S^\k_\o\setminus S_0$ is stationary.
    (Such sets $S_0$ exist by Theorem~\ref{thm:fatsetInL}.)
    Then if $S'$ and $S$ are stationary subsets of $S_\o^\k\setminus S_0$ 
    such that $S'\setminus S$ is stationary, we have $E(T_S)\not\le_B E(T_{S'})$.
  \end{claim}
  \begin{proof}
    Suppose to the contrary that $f\colon 2^{\k}\to 2^\k$ is a Borel reduction from
    $E(T_S)$ to $E(T_{S'})$. 

    The space $[T_S]_\k$ is equipped with the subspace topology inherited from $2^\k$
    and we can define Borel, meager and co-meager subsets of $[T_S]_\k$.
    Note that the meager and co-meager subsets of $[T_S]_\k$ do not coincide with 
    those in $2^\k$, for example
    $[T_S]_\k$ is not meager in $[T_S]_\k$ by Lemma \ref{lem:KurepaCo-meager}
    but meager in~$2^\k$. Now we can define the Baire property
    relativised to $[T_S]_\k$: a set $A\subset [T_S]_\k$ has the Baire property, if
    there exists open $U\subset [T_S]_\k$ such that $U\sd [T_S]_\k$ is meager in $[T_S]_\k$. 
    A standard proof gives that all Borel sets of $[T_S]_\k$ have the Baire property.
    For every $p\in T_{S'}$, the inverse image of $N_{p}$ under $f$ is Borel and so there is
    open $U_p$ such that $U_p\sd f^{-1}N_p$ is meager. 
    Now let $D=[T_S]_\k\setminus \Cup_{p\in 2^{<\k}} U_p\sd f^{-1}N_p$. By 
    Lemma~\ref{lem:KurepaCo-meager} an intersection of $\k$ many dense 
    open sets is non-empty in $[T_S]_\k$ whereas it follows
    that the space is co-meager in itself and $D$ is co-meager.
    So $D\subset [T_S]_\k$ is dense and $f$ is continuous on~$D$.

    By removing one point from $D$, we may assume without loss of generality
    that $f\colon [T_S]_\k\to [T_{S'}]_\k$.

    Let $c(S_0)$ be the set of increasing continuous sequences 
    $(\a_\b)_{\b<\g}$ in $S_0$ with the property that also 
    $\sup_{\b<\g}\a_\b\in S_0$.
    We will now define a function 
    $$\tau\colon c(S_0)\to\k\times \Po(T_S)\times\Po(T_{S'})$$
    by induction on the length of the sequence
    $(\a_\b)_{\b<\g}\in c(S_0)$. The projection of $\tau$ to the
    first coordinate, $\pr_1\circ\tau$ can be thought as a strategy of a player in a climbing game
    (where the players pick ordinals below $\k$ in an increasing way).
    Let $\tau(\es)=(0,\{\eta\rest\d\mid\d<\k\},\{f(\eta)\rest\d\mid \d<\k\})$ 
    where $\eta$ is any element of $[T_S]_\k\cap D$ 
    and suppose $\tau((\a_\d)_{\d<\g+1})$ is defined to be
    $(\a,A,A')$ such that $A$ and $A'$ are subtrees of $T_S$ and $T_{S'}$ respectively
    such that 
    \begin{enumerate}
    \item $\a>\a_\d$ for all $\d<\g+1$,
    \item all the branches of $A$ and $A'$ have length $\k$,
    \item for each branch $\eta$ of $A$, $N_{\eta\restl\a}\cap [A]_\k=\{\eta\}$, i.e. there are no splitting nodes above $\a$, and the same for $A'$,
    \item for each branch $\eta$ of $A$, we have $f[D\cap N_{\eta\restl\a}]\subset N_{\xi\restl\a_\g}$
      for some unique branch $\xi$ of $A'$.\label{condition_4_definition_of_str}
    \end{enumerate}
    Note that the last condition defines an embedding from $[A]_\k$ to $[A']_\k$.
    Now we want to define $\tau((\a_\d)_{\d\le \g+1})=(\b,B,B')$ where 
    $\a_{\g+1}\in S_0 \setminus (\a_\g+1)$. 
    For each branch $\eta$ of $A$, there is a branch $\xi_\eta$ in
    $[T_S]_\k\cap D\cap N_{\eta\restl\a_{\g+1}}$ such that $\xi_\eta(\a_{\g+1}+1)\ne\eta(\a_{\g+1}+1)$
    (for example find $\xi_\eta$ as follows: first note that the function $\xi_\eta'$
    such that $\xi_\eta'(\d)=\eta(\d)$ for $\d\le \a_{\g+1}$ 
    and $\xi_\eta'(\d)=1-\eta(\d)$ for $\d>\a_{\g+1}$ is a branch of $T_S$ (because $\eta$ is a branch
    and $\xi'_{\eta}\rest\b$ is definable from $\eta\rest\b$ for all~$\b$), so by the density of $D$, there
    is $\xi_\eta\in D\cap N_{\xi_{\eta}'\restl\a_{\g+1}+1}$). By condition (3) this branch is 
    new (i.e. not in~$A$).
    Let $B$ be the downward closed subtree of $T_S$ such that
    $[B]_\k=\Cup_{\eta\in A}\{\eta,\xi_\eta\}$ and $B'$ the same for $T_{S'}$ such that
    $[B']_\k=\{f(\eta)\mid \eta\in B\}$.
    Then pick $\b$ high enough so that condition~\eqref{condition_4_definition_of_str} 
    is satisfied for $\a$, $A$ and $A'$ replaced by $\b$, $B$ and $B'$ which is possible by the
    continuity of $f$ on~$D$.
    
    Suppose $\g$ is a limit and $(\a_{\d})_{\d<\g}$ is in $c(S_0)$ and
    $\tau((\a_{\d})_{\d<\e})=(\b_\e,B_\e,B_\e')$ is defined for all $\e<\g$. Let us define 
    $\tau((\a_\e)_{\e<\g})=(\b,B,B')$. Let $\b$ be
    the supremum of $\{\b_\e\mid \e<\g\}$.
    Note that $\Cup_{\e<\g}B_\e$ is a downward closed subset of $T_S$.
    Let $p$ be any branch of length~$\b$ through $\Cup_{\e<\g}B_\e \cap 2^{<\b}$. 
    If there is a branch in $\Cup_{\e<\g}B_\e$ that continues $p$, let $\eta(p)$ be that branch.
    Otherwise, since $\b\notin S$
    and $p\rest\g\in T_S$ for all $\g<\b$, $p$ can be continued
    to some branch $\eta$ in $D\cap T_S$ and we define $\eta(p)$ to be that $\eta$. 
    Let
    $$B=\{\eta(p)\mid p\text{ is a branch through }\Cup_{\e<\g}B_\e\}$$
    and
    $$B'=\{f(\eta)\mid \eta\in B\}.$$
        
    Let $C$ be the cub set of ordinals $\a$ that are closed under $\tau$, in the sense that
    $C$ is the set of those $\a$ such that
    for all sequences $s\in c(S_0)$ that are bounded in $\a$, we have
    $(\pr_1\circ \tau)(s)<\a$, $|(\pr_2\circ \tau)(s)|<\a$ and $|(\pr_3\circ\tau)(s)|<\a$. 
    For each pair of ordinals $(\a_1,\a_2)\in\k$, let $\pi(\a_1,\a_2)$ be the least ordinal such that
    there is an increasing continuous sequence of order type $\a_1$ starting 
    above $\a_2$ with supremum at most $\pi(\a_1,\a_2)$ and let $C_1$ be the cub set of
    ordinals closed under $\pi$. 
    Now by the stationarity of $S'\setminus S$, pick $\a\in C\cap C_1\cap S'\setminus S$.
    Now it is easy to construct a continuous increasing sequence $s$ in $c(S_0)$ of 
    order type $\a$, cofinal in $\a$, and a cofinal sequence $(\g_n)_{n<\o}$ in $\a$ such that
    $s\rest \g_n$ is in $c(S_0)$ for all $n$ and $\tau(s\rest\gamma_n)=(\d_n,A_n,A_n')$ 
    has the following properties:
    \begin{itemize}
    \item $\g_n\le \d_n<\a$,
    \item $A_n$ has at least $2^{\g_n}$ many branches and
    \item $f$ defines a bijection between the branches of $A_n$ and the branches of $A'_n$
    \end{itemize}
    Let $A_\o$ be the tree which consists 
    of those branches $\eta$ of $T_S$ in $D$
    that for every $\d<\a$ there is a branch $\xi$ in $\Cup_{n<\o}[A_n]_\k$ such that
    the common initial segment of $\eta$ and $\xi$ has height at least $\d$. Since $\a\notin S$,
    the number of branches of $A_\o$ is $2^{\a}$. So it means
    that the set $f[[A_\o]_\k]$ must have $2^{\a}$ branches too. The contradiction will follow
    once we show that this implies that $T_{S'}^{\a+1}$ must have $2^{\a}$ elements,
    contradicting Lemma~\ref{lem:SmallLevels}, because $\a\in S'$. 
    But if $\eta$ and $\xi$ are any two branches in $A_\o$, their images must disagree below $\a$
    by the construction, hence $T_{S'}^{\a+1}$ should have at least the same cardinality as~$|A_\o|$.    
  \end{proof}

  \begin{claim}
    If $S\subset S'\subset\k$, then $E(T_{S'})\le_B E(T_S)$.
  \end{claim}
  \begin{proof}
    By the assumption we have $T_{S'}\subset T_{S}$.
    Let $i\colon [T_{S'}]_\k\to [T_S]_\k$ be the inclusion map and let $\xi$
    be a fixed element of $2^\k\setminus[T_S]_\k$.
    For $\eta\in 2^{\k}$, let $f(\eta)=i(\eta)$, if $\eta\in [T_{S'}]_\k$
    and $f(\eta)=\xi$ otherwise.
  \end{proof}

  To prove the Theorem, let $S_0$ be a fat stationary set such that $S^\k_\o\setminus S_0$ is stationary.
  Let $\{S_i\mid i<\k\}$ be a partition of $S^\k_\o\setminus S_0$ into $\k$ many disjoint
  stationary sets. Then by the claims above, the function defined by
  $$A\mapsto E(T_{\Cup_{i\notin A}S_i})$$
  is an embedding $F$ of $\la \Po(\k),\subset\ra$ into $\la \E,\le_B\ra$
  such that for all $A\in \Po(\k)$, we have $F(A)\le_B\id_{2^\k}$ and by the same
  argument as in the proof of Claim~\ref{claim:Something}, we have $\id_{2^\k}\not\le_B F(A)$
\end{proof}

\section{An Antichain Containing the Identity}
\label{sec:Antichain}

In this section $\k$ is regular and uncountable, but not necessarily inaccessible.
We now redefine the meaning of $\sing(\k)$ to be the
set of all $\o$-cofinal ordinals below $\k$ (instead of just cardinals as in the previous section).
Let $T=T_{\sing(\k)}$ (see Definition~\ref{def:Trees}).

As in Lemma \ref{lem:KurepaCo-meager}, $T$ is not meager in itself and we can define the ideal of meager
sets relativised to $T$. In this way, the Borel subsets of $T$ will have the Baire property in~$T$.
Note that $T$ is a meager subset of $2^\k$, so the meager ideal on subsets of $T$ is not a straightforward
restriction of the meager ideal on the subsets of $2^\k$.

\begin{Lemma}\label{lem:ContOnCoM}
  Suppose $f\colon T\to 2^{\k}$ is a Borel function. Then there is a co-meager set $D\subset T$
  such that $f$ is continuous on $D$.
\end{Lemma}
\begin{proof}
  Using Lemma \ref{lem:KurepaCo-meager} as in the beginning of the proof of Claim \ref{claim:Something}
\end{proof}

\begin{Thm}\label{thm:Main2}
  Suppose $V=L$. Then there is an antichain of Borel equivalence relations
  with respect to $\le_B$ of size $2^{\k}$ such that one of the relations is the identity.
\end{Thm}
\begin{proof}
  Let $[T]_\k$ be the set of branches of length $\k$ of $T$. Let $S\subset\k$ be stationary.
  Then let $\eta$ and $\xi$ be $F_S$-equivalent, either if both $\eta$ and $\xi$ 
  are not in $[T]_\k$, or if both $\eta$ and $\xi$ are in $[T]_\k$ \emph{and} are $E_S$-equivalent,
  where $E_S$ is as in~\cite{Kul}:
  $\eta$ and $\xi$ are $E_S$ equivalent if they are $E_0$-equivalent \emph{and} for every $\a\in S$
  there exists $\b<\a$ such that $\forall \g\in \left[\b,\a\right[$, 
  $|\eta(\g)-\xi(\g)|=|\eta(\b)-\xi(\b)|$.

  For a tree $T\subset 2^{<\k}$ and a stationary $S\subset\k$
  define the following game $G(T,S)$ of length $\o$ for two players $\PlOne$ and $\PlTwo$:
  At move $n<\o$, player $\PlOne$ picks a pair $(p^0_{n},p^1_{n})$ of 
  elements of $T$ with $\dom p_n^0=\dom p_n^1$ 
  and then player $\PlTwo$ picks an ordinal $\a_n$ above $\dom p^0_n$. 
  Additionally 
  the following conditions should be satisfied by the moves of player~$\PlOne$:
  \begin{enumerate}
  \item $p_{n-1}^0\subset p_{n}^0$ and $p_{n-1}^1\subset p_{n}^1$,
  \item $\dom p_n^0=\dom p_n^1>\a_{n-1}$.
  \end{enumerate}
  Suppose that $(p_n^i)_{n<\o}$ for $i\in\{0,1\}$ are the sequences obtained in this way by player $\PlOne$.
  Player $\PlTwo$ wins, if player $\PlOne$ didn't follow the rules, or else
  $\Cup_{n<\o}p_n^0$ and $\Cup_{n<\o}p_n^1$ are both in $T$ and $\sup_{n<\o}\dom p_n^0\in S$.

  \begin{claim}
    Suppose $S\subset \sing(\k)$ is stationary and $T$ 
    is the weak Kurepa tree defined above. 
    Then Player $\PlOne$ has no winning strategy in $G(T,S)$.
  \end{claim}
  \begin{proof}
    Suppose $\tau$ is a strategy of Player $\PlOne$. 
    Let $M$ be an elementary submodel of $(L_{\k^+},\tau,S,\in)$ of size $\k$ such that
    $M\cap L_{\k}$ is transitive and
    $M\cap \k=\a$ for some $\a\in S$. Let $f(\a)$ be the least ordinal such that $\a$ is
    singular in $L_{f(\a)}$ and let $r=(r_i)_{i<\o}$ be some cofinal sequence in $\a$ in $L_{f(\a)}$.
    Now player $\PlTwo$ can play against $\tau$ in $L_{f(\a)}$ 
    towards $\a$ using $r$. The replies of $\PlOne$ will be in fact in $M$ and the
    eventual sequences $(p^k_n)_{n<\o}$,
    $k\in\{0,1\}$, constructed by $\PlOne$ will be in $L_{f(\a)}$ and so by definition
    $\Cup_{i<\o}p^k_n$ will be in $T$ and so player $\PlTwo$ wins this game.
  \end{proof}

  \begin{claim}
    If $S'\setminus S$ is $\o$-stationary, then $F_{S}$ is not Borel-reducible to $F_{S'}$.
  \end{claim}
  \begin{proof}
    The argument is as in~\cite{Kul}. Suppose $f$ is a Borel
    function from $[T]_\k$ to $[T]_\k$ which reduces $F_{S}$ to $F_{S'}$ for
    some stationary $S$ and $S'$ such that $S'\setminus S$ is
    stationary. We will derive a contradiction.  By Lemma~\ref{lem:ContOnCoM}  
    there exists a sequence $(D_{i})_{i<\k}$ of dense open sets
    such that $f$ is continuous on the co-meager set $D=\Cap_{i<\k}D_i$. 
    We will now define a strategy of
    player $\PlOne$ in $G(T,S'\setminus S)$ such that if it is not a winning
    strategy, then the contradiction is achieved, so we are done by the
    claim above.
    
    The strategy is as follows. At the first move, player $\PlOne$ picks
    a function $\eta\in 2^{\k}$ with the property that both $\eta$ and $1-\eta$
    are branches of $T$ and in $D$. Since $\eta$ and $1-\eta$ are non-equivalent in $F_S$,
    $f(\eta)$ and $f(1-\eta)$ are non-equivalent in $F_{S'}$.
    So there is a point $\a$ such that $f(\eta)(\a)\ne f(1-\eta)(\a)$.
    Player $\PlOne$ then finds $\a_0$ such that $f[D\cap N_{\eta\restl\a_0}]\subset N_{f(\eta)\rest(\a+1)}$
    and $f[D\cap N_{(1-\eta)\restl\a_0}]\subset N_{f(1-\eta)\restl (\a+1)}$. 
    The first move is the pair $(p_0^0,p_0^1)$ where
    $p_0^0=\eta\restl\a_0$ and $p_0^1=(1-\eta)\restl\a_0$.
    Additionally player $\PlOne$ keeps in mind the elements 
    $q_0^0=f(\eta)\rest(\a+1)$ and $q_0^1=(f(1-\eta)\restl (\a+1))$. Suppose the players have played
    $n$ moves and $(\b_0,\dots,\b_n)$ are the ordinals picked by player $\PlTwo$ and 
    $((p_0^0,p_0^1),\dots,(p_n^0,p_n^1))$
    the pairs picked by player $\PlOne$. Player $\PlOne$ has also constructed a sequence
    $(q_i^0,q_i^1)_{i\le n}$. If $n$ is even, then player $\PlOne$ extends
    $p_n^0$ and $p_n^1$ into branches $\eta$ and $\xi$ of $T$ 
    such that $\eta(\a)=\xi(\a)$ implies $\a<\dom p_n^0=\dom p_n^1=\a_n$
    and such that $\eta$ and $\xi$ are both in $D$. By the induction
    hypothesis $f(\eta)$ extends $q_n^0$ and $f(\xi)$
    extends $q_n^1$, so we can find $\b_n'>\b_n$ such that the
    continuations $q_{n+1}^0=f(\eta)\rest\b_n'$ and $q_{n+1}^1=f(\xi)\rest\b_{n'}$ are 
    of equal length and for some $\b\in \dom q_{n+1}^0\setminus \b_n$ with $q_{n+1}^0(\b)\ne q_{n+1}^1(\b)$
    (if such $\b'_n$ does not exist, 
    then it implies that $q_{n}^0$ and $q_{n}^{1}$ cannot be extended 
    to $F_{S'}$-equivalent branches whereas $p_{n}^0$ and $p_{n}^1$ can be 
    extended to $F_S$-equivalent branches, which would be a contradiction). 
    Then player $\PlOne$ finds 
    an $\a_{n+1}>\b_n'$ such that, denoting $p_{n+1}^0=\eta\rest\a_{n+1}$
    and $p_{n+1}^{1}=\xi\rest\a_{n+1}$, we have 
    $$f[D\cap N_{p_{n+1}^0}]\subset N_{q_{n+1}^0}$$ 
    and
    $$f[D\cap N_{p_{n+1}^1}]\subset N_{q_{n+1}^1}.$$
    The pair $(p^0_{n+1},p^1_{n+1})$ is the next move. 
    If $n$ is odd, then player $\PlOne$ proceeds in the same way, 
    but with the only differences that now
    he picks $\eta$ and $\xi$ such that $\eta(\a)=\xi(\a)$ for all $\a>\dom \a_n$
    and finds $q_{n+1}^0$ and $q_{n+1}^1$ such that $q_{n+1}^0(\b)=q_{n+1}^{1}(\b)$ for some
    $\b\in \dom q_{n+1}^0\setminus \b_n$.
    This describes the strategy.
    
    If player $\PlTwo$ beats this strategy in $G(T,S'\setminus S)$, it means
    that the limit of her moves, which is the same as the limit of the sequence 
    $(\dom p_n^i)_{n<\o}$, $i\in\{0,1\}$, is in $S'$ and not in $S$. So by looking at 
    the things that player $\PlOne$ has constructed, we note that $p_\o^0=\Cup_{n<\o}p_n^0$ 
    and $p_\o^1=\Cup_{n<\o}p_n^1$ can be extended
    to equivalent branches on the side of $F_S$, but 
    $q_\o^0=\Cup_{n<\o}q_n^0$ and $q_\o^1=\Cup_{n<\o}q_n^1$ 
    cannot be extended (in $D$) to equivalent branches
    on the range side $F_{S'}$ which is a contradiction, because 
    $f[D\cap N_{p_\o^i}]\subset N_{q_\o^i}$, $i\in\{0,1\}$.
  \end{proof}
  
  To prove the Theorem, let $(S_i)_{i<\k}$ be a partition of $S^\k_\o$ into disjoint stationary pieces.
  Then let $\A$ be a maximal antichain in $\Po(\k)$ (a set of size $2^\k$ of subsets of $\k$ incomparable
  under inclusion) 
  and define $G\colon\A\to \E$
  by $G(A)=F_{\Cup_{i\notin A}S_i}$. Then $G[\A]$ is an antichain by the claims above. 
  Every element of this antichain
  is incomparable with identity: identity is not reducible to any of them, because
  of the small levels guaranteed by the weak Kurepa tree $T$. On the other hand any of the relations
  is not reducible to $\id$ because of the $E_0$-component: the equivalence classes are dense in $T$ which
  violates the continuity of any reduction even on an (arbitrary) co-meager set.
\end{proof}

\bibliography{ref}{}
\bibliographystyle{alpha}


\end{document}